\documentclass[a4paper,10pt]{article}
\usepackage[english]{babel}
\usepackage[utf8]{inputenc}
\usepackage{hyperref}
\usepackage{geometry}
\usepackage{authblk, verbatim}

\title{Modular and fractional $L$-intersecting \\families of vector spaces}
\author[1]{Rogers Mathew \footnote{This author was supported by a grant from the Science and Engineering Research
Board, Department of Science and Technology, Govt. of India (project number: MTR/2019/000550).}}
\author[2]{Tapas Kumar Mishra}
\author[3]{Ritabrata Ray}
\author[4]{Shashank Srivastava}

\affil[1]{Department of Computer Science and Engineering, \protect\\
Indian Institute of Technology Hyderabad, India. \protect\\
rogers@iith.ac.in}
\affil[2]{Department of Computer Science and Engineering, \protect\\
National Institute of Technology Rourkela, India. \protect \\
mishrat@nitrkl.ac.in}
\affil[3]{Department of Electrical \& Computer Engineering, \protect\\
Cornell University, Ithaka, NY 14853, USA. \protect\\
rayritabrata96@gmail.com}
\affil[4]{Toyota Technological Institute at Chicago, Chicago, IL  60637, USA. \protect\\
shashanks@ttic.edu}
\usepackage{theoremref}
\usepackage{amsmath, xcolor}
\usepackage{amsfonts}
\usepackage{amssymb}
\usepackage{graphicx}
\usepackage{subcaption}
\usepackage{float}
\usepackage[nottoc,numbib]{tocbibind}
\usepackage{amsthm}	
\usepackage{enumerate}

\theoremstyle{definition}
\newtheorem{theorem}{Theorem}[section]
\newtheorem{lemma}[theorem]{Lemma}
\newtheorem{proposition}[theorem]{Proposition}
\newtheorem{corollary}[theorem]{Corollary}
\newtheorem{definition}[theorem]{Definition}

\newtheorem{example}{Example}

\newcommand\qbin[3]{\left[\begin{matrix} #1 \\ #2 \end{matrix} \right]_{#3}}

\date{}

\begin{document}
\maketitle
\begin{abstract}
In the first part of this paper, we prove the following theorem which is the $q$-analogue of a generalized modular Ray-Chaudhuri-Wilson Theorem shown in [Alon, Babai, Suzuki, J. Combin. Theory Series A, 1991]. It is also a generalization of the main theorem in [Frankl and Graham, European J. Combin. 1985] under certain circumstances. 
\\ $\bullet$ Let $V$ be a vector space of dimension $n$ over a finite field of size $q$. Let $K = \{k_1, \ldots , k_r\},L = \{\mu_1, \ldots , \mu_s\}$ be two disjoint subsets of $\{0,1, \ldots , b-1\}$ with $k_1 < \cdots < k_r$.  
   Let 
   $\mathcal{F} = \{V_1,V_2,\ldots,V_m\}$ be a family of subspaces of $V$ such that $\forall i \in [m]$, dim($V_i$) $\equiv k_t \pmod b$, for some $k_t \in K$; for every distinct $i,j \in [m]$, dim($V_i \cap V_j$) $ \equiv \mu_t \pmod b$, for some $\mu_t \in L$. Moreover, it is given that neither of the following two conditions hold:
   \begin{enumerate}[(i)]
   \item $q+1$ is a power of 2, and $b=2$
   \item $q=2, b=6$
   \end{enumerate}
 Then,
  \begin{equation*}
    		|\mathcal{F}| \leq \begin{cases}
    											\scalebox{0.7}{ $N(n,s,r,q)$}, \hskip 1cm \mbox{ if }\left(s+k_r \leq n \mbox { and } r(s-r+1) \leq b-1\right) \mbox{ or } (s < k_1 + r)  \\
    											\scalebox{0.7}{ $N(n,s,r,q) + \sum\limits_{t \in [r]}\qbin{n}{k_t}{q}$}, \hskip 1cm \text{otherwise}
											\end{cases}    		
    		\end{equation*}
, where $N(n,s,r,q) := \qbin{n}{s}{q} + \qbin{n}{s-1}{q} + \cdots + \qbin{n}{s-r+1}{q}.$  

In the second part of this paper, we prove $q$-analogues of results on a recent notion called  \emph{fractional $L$-intersecting family} of sets for families  of subspaces of a given vector space over a finite field of size $q$. We use the above theorem to obtain a general upper bound to the cardinality of such families. We give an improvement to  this general upper bound in certain special cases. 
\end{abstract}      
  \begin{section}{Introduction}
   Let $[n]$ be the set of all natural numbers from $1$ to $n$. A family $\mathcal{F}$ of subsets of $[n]$ is called \emph{intersecting} if every set  in $\mathcal{F}$  has a non-empty intersection with every other set in $\mathcal{F}$. One of the earliest studies on intersecting families dates back to the famous Erd\H{o}s-Ko-Rado Theorem \cite{erdos1961intersection} about maximal uniform intersecting families.  Ray-Chaudhuri and Wilson \cite{ray-chaudhuri1975} introduced the notion of  $L$-intersecting families. Let $L = \{l_1, \ldots , l_s\}$ be a set of non-negative integers. A family $\mathcal{F}$ of subsets of $[n]$ is said to be \emph{$L$-intersecting} if for every distinct $F_i, F_j$ in $\mathcal{F}$, $|F_i \cap F_j| \in L$. 
  The Ray-Chaudhuri-Wilson Theorem states that if $\mathcal{F}$ is $t$-uniform (that is, every set in $\mathcal{F}$ is $t$-sized), then $|\mathcal{F}| \leq {n \choose s}$. This bound is tight as shown by the set of all $s$-sized subsets of $[n]$ with $L=\{0, \ldots , s-1\}$. Frankl-Wilson Theorem \cite{Frankl1981} extends this to non-uniform families by showing that $|\mathcal{F}| \leq \sum\limits_{i=0}^{s}{n \choose i}$, where $\mathcal{F}$ is any family of subsets of $[n]$ that is $L$-intersecting. The collection of all the subsets of $[n]$ of size at most $s$ with $L= \{0, \ldots s-1\}$ is a tight example to this bound. The first proofs of these theorems were based on the technique of higher incidence matrices. Alon, Babai, and Suzuki in \cite{ALON1991165} generalized the Frankl-Wilson Theorem using a proof that operated on spaces of multilinear polynomials. They showed that if the sizes of the sets in $\mathcal{F}$ belong to $K=\{k_1, \ldots , k_r\}$ with each $k_i > s-r$, then $|\mathcal{F}| \leq {n \choose s} + \cdots + {n \choose s-r+1}$. A modular version of the Ray-Chaudhuri-Wilson Theorem was shown in \cite{FranklW81}. This result was generalized in \cite{ALON1991165}. See \cite{liu2014set} for a survey on $L$-intersecting families. 
  
  Researchers have also been working on similar intersection theorems for subspaces of a given vector space over a finite field. Hsieh \cite{hsieh1975intersection}, and Deza and Frankl \cite{deza1983erdos} showed Erd\H{o}s-Ko-Rado type theorems for subspaces. 
  Let $V$ be a vector space of dimension $n$ over a finite field of size $q$. The number of $d$-dimensional subspaces of $V$ is given by the $q$-binomial coefficient $\qbin{n}{d}{q} = \frac{(q^n-1)(q^{n-1}-1)\cdots(q^{n-d+1}-1)}{(q^d-1)(q^{d-1}-1)\cdots(q-1)}$. The following theorem which is a $q$-analog of the Ray-Chaudhuri-Wilson Theorem by considering families of subspaces instead of subsets is due to \\ \cite{FRANKL1985183}.	
  \begin{theorem}\thlabel{Frankl_Graham_subspace}[Theorem 1.1 in \cite{FRANKL1985183}]
   Let $V$ be a vector space over of dimension $n$ over a finite field of size $q$. Let $\mathcal{F} = \{V_1,V_2,\ldots,V_m\}$ be a family of subspaces of $V$ such that dim($V_i$) $ = k$, $\forall i \in [m]$. Let $0 \leq \mu_1 < \mu_2 < \cdots < \mu_s < b$ be integers such that $k \not\equiv \mu_t \pmod b$, for any $t$. For every $1 \leq i < j \leq m$, dim($V_i \cap V_j$) $ \equiv \mu_t \pmod b$, for some $t$.  
   Then,
   \begin{equation*}
    |\mathcal{F}| \le \qbin{n}{s}{q}
   \end{equation*}
   except possibly for $q=2,b=6,s \in \{3,4\}$. 
  \end{theorem}
 \begin{example} [Remark 3.2 in \cite{FRANKL1985183}]\label{exam:frankl_graham}
 Let $n = k+s$. Let $\mathcal{F}$ be the family of all the $k$-dimensional subspaces of $V$, where $V$ is an $n$-dimensional vector space over a finite field of size $q$. Observe that, for any two distinct $V_i, V_j \in \mathcal{F}$, $k-s  \leq dim(V_i \cap V_j) \leq  k-1$. This is a tight example for Theorem \ref{Frankl_Graham_subspace}. 
 \end{example}
Alon et al. in \cite{ALON1991165} proved a generalization of the non-modular version of the above theorem. This result was subsequently strengthened  in \cite{liu2018common}.  
  
Our paper is divided into two parts. In Part A of the paper, we prove the following theorem which is a generalization of Theorem \ref{Frankl_Graham_subspace} due to Frankl and Graham under certain circumstances. It is also the $q$-analogue of a generalized modular Ray-Chaudhuri-Wilson Theorem shown in \cite{ALON1991165}. We assume that $\qbin{a}{b}{q} = 0$, when $b<0$ or $b>a$. Let 
$$N(n,s,r,q) := \qbin{n}{s}{q} + \qbin{n}{s-1}{q} + \cdots + \qbin{n}{s-r+1}{q}.$$
\begin{theorem}
\thlabel{th:1}
   Let $V$ be a vector space of dimension $n$ over a finite field of size $q$. Let $K = \{k_1, \ldots , k_r\},L = \{\mu_1, \ldots , \mu_s\}$ be two disjoint subsets of $\{0,1, \ldots , b-1\}$ with $k_1 < \cdots < k_r$.  
   Let 
   $\mathcal{F} = \{V_1,V_2,\ldots,V_m\}$ be a family of subspaces of $V$ such that $\forall i \in [m]$, dim($V_i$) $\equiv k_t \pmod b$, for some $k_t \in K$; for every distinct $i,j \in [m]$, dim($V_i \cap V_j$) $ \equiv \mu_t \pmod b$, for some $\mu_t \in L$. Moreover, it is given that neither of the following two conditions hold:
   \begin{enumerate}[(i)]
   \item $q+1$ is a power of 2, and $b=2$
   \item $q=2, b=6$
   \end{enumerate}
 Then,
  \begin{equation*}
    		|\mathcal{F}| \leq \begin{cases}
    											\scalebox{0.7}{ $N(n,s,r,q)$}, \hskip 1cm \mbox{ if }\left(s+k_r \leq n \mbox { and } r(s-r+1) \leq b-1\right) \mbox{ or } (s < k_1 + r)  \\
    											\scalebox{0.7}{ $N(n,s,r,q) + \sum\limits_{t \in [r]}\qbin{n}{k_t}{q}$}, \hskip 1cm \text{otherwise.}
											\end{cases}    		
    		\end{equation*}
  \end{theorem}
 In Part B, we study a notion of fractional $L$-intersecting families which was introduced in \cite{BalRgmTap}. We say a family $\mathcal{F} = \{F_1, F_2, \ldots, F_m\}$ of subsets of $[n]$ is a \emph{fractional $L$-intersecting family}, where $L$ is a set of irreducible fractions between $0$ and $1$, if $\forall i,j \in [m], i \neq j$, $\frac{|F_i \cap F_j|}{|F_i|} \in L$  or  $\frac{|F_i \cap F_j|}{|F_j|} \in L$. In this paper, we extend this notion from subsets to subspaces of a vector space over a finite field. 
 \begin{definition}
 Let $L= \{\frac{a_1}{b_1}, \ldots , \frac{a_s}{b_s}\}$ be a set of positive irreducible fractions, where every $\frac{a_i}{b_i} < 1$. Let $\mathcal{F} = \{V_1, \ldots , V_m\}$ be  a family of subspaces of a vector space $V$ over a finite field. We say $\mathcal{F}$ is a \emph{fractional $L$-intersecting family of subspaces} if for every two distinct $i,j \in [m]$, $\frac{dim(V_i \cap V_j)}{dim(V_i)} \in L$  or  $\frac{dim(V_i \cap V_j)}{dim(V_j)} \in L$. 
 \end{definition}
 When every subspace in $\mathcal{F}$ is of dimension exactly $k$, it is an $L'$-intersecting family where $L'= \{\frac{a_1k}{b_1}, \ldots , \frac{a_sk}{b_s}\}$. Applying Theorem \ref{Frankl_Graham_subspace}, we get $|\mathcal{F}| \leq \qbin{n}{s}{q}$. A tight example to this is the collection of all $k$-dimensional subspaces of $V$ with $L=\{\frac{0}{k}, \ldots , \frac{k-1}{k}\}$.  However, the problem of bounding the cardinality of a fractional $L$-intersecting family of subspaces becomes more  interesting when $\mathcal{F}$ contains subspaces of various dimensions. In Part B of this paper, we obtain upper bounds for the cardinality of a fractional $L$-intersecting family of subspaces that are $q$-analogs of the results in \cite{BalRgmTap}. With the help of Theorem \ref{th:1} that we prove in Part A, we obtain the following result in Part B. 
 \begin{theorem}
 \thlabel{thm:frac_general_thm}
 Let $L = \{ \frac{a_1}{b_1},\frac{a_2}{b_2},\ldots,\frac{a_s}{b_s}\}$ be a collection of  positive irreducible fractions, where every $\frac{a_i}{b_i} < 1$. Let $\mathcal{F}$ be a fractional $L$-intersecting family of subspaces of a vector space $V$ of dimension $n$ over a finite field of size $q$. Let $t = \max\limits_{i \in [s]} b_i$, $g(t,n) = \frac{2(2t+\ln n)}{\ln (2t+\ln n)}$, and $h(t,n) = \min(g(t,n),\frac{\ln n}{\ln t})$. Then, 
 $$|\mathcal{F}| \leq 2g(t,n)h(t,n)\ln(g(t,n)) \qbin{n}{s}{q} + h(t,n) \sum\limits_{i=1}^{s-1} \qbin{n}{i}{q}.$$ Further, if $2g(t,n)\ln(g(t,n)) \leq n+2$, then 
 $$|\mathcal{F}| \leq 2g(t,n)h(t,n)\ln(g(t,n)) \qbin{n}{s}{q}.$$
 	
 	\end{theorem}
 \begin{example} \label{exam:frac_gen} Let $s$ be a constant,  $L= \{\frac{0}{s},\frac{1}{s},\ldots,\frac{s-1}{s}\}$, and  $\mathcal{F}$ be the family of all the $s$-sized subspaces of $V$. Clearly, $\mathcal{F}$ is a fractional $L$-intersecting family  showing that the bound in Theorem \ref{thm:frac_general_thm} is asymptotically tight up to a multiplicative factor of $\frac{\ln^2 n}{\ln \ln n}$.	
\end{example}
We improve the bound obtained in Theorem  \ref{thm:frac_general_thm} for the special case when $L=\{\frac{a}{b}\}$, where $b$ is a prime. 

\begin{theorem}
\thlabel{th:61}
   Let $L=\{\frac{a}{b}\}$, where $\frac{a}{b}$ is a positive irreducible fraction less than $1$ and $b$ is a prime. Let  $\mathcal{F}$ be a fractional $L$-intersecting family of subspaces of a vector space $V$ of dimension $n$ over a finite field of size $q$. Then, we have $ |\mathcal{F}| \leq (b-1)( \qbin{n}{1}{q} +1) \lceil \frac{\ln n}{\ln b} \rceil + 2$.
  \end{theorem}
 \begin{example}
 \label{exam:frac_spe}
 Let $L=\{\frac{1}{2}\}$. Let $V$ be a vector space of dimension $n$ over a finite field of size $q$. Let $\{v_1, v_2, \ldots , v_n\}$ be a basis of $V$. Let $V' := span(\{v_2, \ldots , v_n\})$ be an $(n-1)$-dimensional subspace of $V$. Let $\mathcal{F}$ be the set of all $\qbin{n-1}{1}{q}$ $2$-dimensional subspaces of $V$ each of which is obtained by a span of $v_1$ and each of the  $\qbin{n-1}{1}{q}$ $1$-dimensional subspaces of $V'$. This example shows that when $b$ and $q$ are constants, the bound in Theorem \ref{th:61} is asymptotically tight up to a multiplicative factor of $\ln n$. 
 \end{example} 
  
  


\end{section}

\begin{center}\section*{PART A: Generalized modular RW Theorem for subspaces}\end{center}
As mentioned before, in this part we prove \thref{th:1}. The  approach followed here is similar to the approach used in proving Theorem 1.5, a generalized modular Ray-Chaudhuri-Wilson Theorem for subsets, in \cite{ALON1991165}. We start by stating the Zsigmondy's Theorem which will be used in the proof of Theorem \ref{th:1}. 
\begin{theorem}[\cite{zsigmondy1892theorie}]
\thlabel{thm:Zsigmondy}
For any $q,b \in \mathbb{N}$, there exists a prime $p$ such that $q^b \equiv 1 \pmod p, q^i \not\equiv 1 \pmod p ~\forall i,~0<i<b$, except when 
(i) $q+1$ is a power of $2$, $b=2$, or 
(ii) $q=2,b=6$. 
\end{theorem}  
\subsection{Notations used in PART A}
Unless defined explicitly, in the rest of Part A, the symbols $K = \{k_1, \ldots , k_r\}$, $r$, $L = \{\mu_1, \ldots , \mu_s\}$, $s$, $q$, $V$, $\mathcal{F}$, $n$, $b$, $m$, and $V_1, \ldots , V_m$ are defined as they are defined in Theorem \ref{th:1}. 
   We shall use $U \subseteq V$ to denote that $U$ is a subspace of $V$. Using Zsigmondy's Theorem, we find a prime $p$ so that $q^i \not\equiv 1 \pmod p$ for $ 0 < i < b$ and $q^b \equiv 1 \pmod p$. This is possible except in the two cases specified  in \thref{thm:Zsigmondy}. We ignore these two cases from now on in the rest of Part A. 

\subsection{M{\"o}bius inversion over the subspace poset}
Consider the partial  order defined on the set of subspaces of the vector space $V$ over a finite field of size $q$ under the `containment' relation. Let $\alpha$ be a  function from the set of subspaces of $V$ to $\mathbb{F}_p$. A function  $\beta$ from the set of subspaces of $V$ to $\mathbb{F}_p$ is the \emph{zeta transform} of $\alpha$ if $\forall W \subseteq V,~\beta(W) = \sum_{U \subseteq W} \alpha(U)$. Then, applying the \emph{M{\"o}bius inversion formula} we get $\forall W \subseteq V$, $\alpha(W) = \sum_{ \substack{U\subseteq W}} \mu(U,W) \beta(U)$, where $\alpha$ is called the \emph{M{\"o}bius  transform} of $\beta$ and $\mu(U,W)$ is  the \emph{M{\"o}bius function} for the subspace poset. In the proposition below, we show that the \emph{M{\"o}bius function} for the subspace poset is defined as 
\begin{equation*}
    		\mu(X,Y) =   \begin{cases}
    											\scalebox{0.7}{ $(-1)^dq^{{d \choose 2}}$},  \mbox{ if }X \subseteq Y  \\
    											\scalebox{0.7}{ $0$}, \hskip 1cm \text{otherwise}
											\end{cases}    		
    		\end{equation*}
, $\forall X,Y \subseteq V$ with $d = dim(Y) - dim(X)$. 
\begin{proposition}\thlabel{prop:1}
	Let $\alpha$ and $\beta$ be functions from the set of subspaces of $V$ to $\mathbb{F}_p$. 
	Then, $\forall\ W \subseteq V$,
	   
   \begin{gather*}
    \beta(W) = \sum_{U\subseteq W} \alpha(U) \iff \alpha(W) = \sum_{ \substack{U\subseteq W \\ d= dim(W)-dim(U)}} (-1)^d q^{\frac{d(d-1)}{2}} \beta(U)
   \end{gather*}
   \end{proposition}
   \begin{proof}
   Proof in Appendix \ref{append_proof_of_prop_inversion_formula}. 
   \end{proof}   
   
   \begin{definition}
    Given two subspaces $U$ and $W$ of the vector space $V$, we define their union space $U \cup W$ as the span of union of sets of vectors in $U$ and $W$.
   \end{definition}
   \begin{proposition}
   \label{prop:mob_inv_gen}
    Let $\alpha$ and $\beta$ be functions as defined in \thref{prop:1}. 
    Then, $\forall\ W,Y$ such that $W \subseteq Y \subseteq V$,
    \begin{equation*}
     \sum_{\substack{T :\ W \subseteq T \subseteq Y \\ d=dim(Y)-dim(T)}} (-1)^d q^{\frac{d(d-1)}{2}} \beta(T) = \sum_{U:\ U \cup W = Y} \alpha(U)
    \end{equation*}
   \end{proposition}
   \begin{proof}
    Proof in Appendix \ref{append_proof_of_prop_mob_inv_gen}
    \end{proof}
   \begin{corollary}\thlabel{cor:1}
    For any non-negative integer $g$, the following are equivalent for functions $\alpha$ and $\beta$ defined in \thref{prop:1}:
    \begin{enumerate}[(i)]
     \item $\alpha(U) = 0$, $\forall U \subseteq V$ with $dim(U) \geq g$.
     \item $\sum\limits_{\substack{W \subseteq T \subseteq Y \\ d=dim(Y)-dim(T)}} (-1)^d q^{\frac{d(d-1)}{2}} \beta(T) = 0$, $\forall W,Y \subseteq V$ with $dim(Y) - dim(W) \geq g$.
    \end{enumerate}
   \end{corollary}
   
   \begin{definition}
   Let $H = \{h_1, h_2, \ldots , h_t\}$ be a subset of $\{0,1,\ldots , n\}$ where $h_1 < h_2 < \cdots < h_t$. We say $H$ has a \emph{gap} of size $\geq g$ if either $h_1 \geq g-1$, $n-h_t \geq g-1$, or $h_{i+1} - h_i  \geq g$ for some $i \in [t-1]$. 
   \end{definition}
   
   \begin{lemma}\thlabel{lem:3}
    Let $\alpha$ and $\beta$ be functions as in \thref{prop:1}. Let $H \subseteq \{0,1,\ldots,n\}$ be a set of integers and $g$ an integer, $0 \leq g \leq n$. Moreover, we have the following conditions:
    \begin{enumerate}[(i)]
     \item $\forall U\subseteq V$, we have $\alpha(U) = 0$ whenever $dim(U) \geq g$.
     \item $\forall T \subseteq V$, we have $\beta(T) = 0$ whenever $dim(T) \notin H$.
     \item $H$ has a gap $\geq g+1$.
    \end{enumerate}
    Then, $\alpha = \beta = 0$.

   \end{lemma}
  
    \begin{proof}
    Let $H=\{h_1,h_2,\ldots,h_{|H|}\}$. Suppose the gap is between $h_i$ and $h_{i-1}$, or between $-1$ and $h_1$, then we may assume $\exists i$ such that $h_i \in H$ and $h_i -j \notin H$ for $ 1 \leq j \leq g$ and $h_i-g \geq 0$. Choose any two subspaces, say $U$ and $W$, of $V$ of dimensions $h_i$ and $h_i-g$, respectively. Since  $dim(U) \geq g$, $\alpha(U) = 0$. We know from \thref{cor:1} that 
    
    \begin{equation*}
      \sum\limits_{\substack{W \subseteq T \subseteq U \\ d=dim(U)-dim(T)}} (-1)^d q^{\frac{d(d-1)}{2}} \beta(T) = 0
    \end{equation*}
But whenever $dim(T)<h_i$, it lies between $h_i-g$ and $h_i-1$, and hence $\beta(T)=0$.
Then,
    \begin{equation*}
     \sum\limits_{\substack{W \subseteq T \subseteq U \\ d=dim(U)-dim(T)}} (-1)^d q^{\frac{d(d-1)}{2}} \beta(T) = \beta(U) = 0
    \end{equation*}
Since our choice of $U$ was arbitrary, we may conclude that $\beta(U)=0, \forall U \subseteq V$ with $dim(U)=h_i$. Thus, we can remove $h_i$ from the set $H$, and then use the same procedure to further reduce the size of $H$ till it is an empty set. If $H$ is empty, $\beta(U)=0,\ \forall U \subseteq V$, giving $\alpha(U) = \beta(U)=0$ as required.
    
    Now suppose the gap was between $h_{|H|}$ and $n+1$. In this case, we take $U$ of dimension $h_{|H|}$ and $W$ of dimension $h_{|H|}+g$ to show that $\beta(U)=0$, and remove $h_{|H|}$ from $H$. Note that removing a number from the set $H$ can never reduce the gap.
    \end{proof}

  \subsection{Defining functions  $f^{x,y}$ and $g^{x,y}$}\label{sec:3}
  Consider all the subspaces of the vector space $V$. We can impose an ordering on the subspaces of same dimension, and use the natural ordering across dimensions, so that every subspace can be uniquely represented by a pair of integers $\langle d,e\rangle$, indicating that it is the $e^{th}$ subspace of dimension $d$, $0 \leq d \leq n$, $1 \leq e \leq \qbin{n}{d}{q}$. Let us call that subspace $V_{d,e}$. 
   Let $S$ be the number of subspaces of $V$ of dimension at most $s$, that is, $S = \sum_{t=0}^{s} \qbin{n}{t}{q}$. Let each subspace $V_{d,e}$ of dimension at most $s$ be represented as a $0$-$1$ containment vector $v_{d,e}$ of $S$ entries, each entry of the vector denoting whether a particular subspace of dimension $\leq s$ is contained in $V_{d,e}$ or not.
      \begin{equation*}
    v_{d,e}^{x,y} = \begin{cases}
                     1 \text{, if }V_{x,y}\text{ is a subspace of }V_{d,e} \\
                     0 \text{, otherwise}
                    \end{cases}
   \end{equation*}
   The vector $v_{d,e}$ is made up of $v_{d,e}^{x,y}$ values for $0 \leq x \leq s$, $1 \leq y \leq \scalebox{0.5}{$\qbin{n}{x}{q}$}$, making it a vector of size $S$. For $0 \leq x \leq s, 1 \leq y \leq \scalebox{0.7}{ $\qbin{n}{x}{q}$ }$ we define functions $f^{x,y}: \mathbb{F}_2 ^S \rightarrow \mathbb{F}_p$ as
   \begin{equation*}
    f^{x,y}(v) = f^{x,y}(v^{0,1},v^{1,1}, \ldots , v^{1,\scalebox{0.5}{ $\qbin{n}{1}{q}$ }}, \ldots , v^{s,1}, \ldots, v^{s,\scalebox{0.5}{ $\qbin{n}{s}{q}$ }}) := v^{x,y}.
   \end{equation*}
   For $0 \leq x \leq s-r, 1 \leq y \leq \scalebox{0.7}{ $\qbin{n}{x}{q}$ }$, we define functions $g^{x,y}: \mathbb{F}_2 ^S \rightarrow \mathbb{F}_p$ as
   \begin{equation*}
    g^{x,y}(v) = f^{x,y}(v)\prod_{t \in [r]}(\sum_{j=1}^{ \scalebox{0.7}{$ \qbin{n}{1}{q}\text{ }$} }v^{1,j} - \qbin{k_t}{1}{q})
   \end{equation*}
  Let $\Omega$ denote $\mathbb{F}_2^S$. The functions $f^{x,y}$ and $g^{x,y}$ reside in the space $\mathbb{F}_p^\Omega$.

\subsection{Swallowing trick: linear independence of functions $f^{x,y}$ and $g^{x,y}$}\label{sec:4}
   \begin{lemma}\thlabel{lem:4}
   Let $s+k_r \leq n$ and $r(s-r+1) \leq b-1$. The functions $g^{x,y}$, $0 \leq x \leq s-r, 1 \leq y \leq \scalebox{0.7}{ $\qbin{n}{x}{q}$ }$, are linearly independent in the function space $\mathbb{F}_p^{\Omega}$ over $\mathbb{F}_p$.
  \end{lemma}
  
  \begin{proof}
   We wish to show that the only solution to $\sum\limits_{\substack{0 \leq x \leq s-r \\ 1\leq y \leq \scalebox{0.7}{$\qbin{n}{x}{q}$}}} \alpha^{x,y}g^{x,y} = 0$ is the trivial solution $ \alpha^{x,y}=0,~\forall x, y$.
   We define function $\alpha$ from the set of all subspaces of $V$ to $\mathbb{F}_p$ as:
   \begin{equation*}
    \alpha(V_{d,e}) = \begin{cases}
                       \alpha^{d,e},\text{ if }0 \leq d \leq s-r \\
                       0,\text{\quad if }d > s-r
                      \end{cases}
   \end{equation*}
We show that functions $\alpha$ and $\beta(U) := \sum\limits_{T\subseteq U} \alpha(T)$ satisfy the conditions of \thref{lem:3}, thereby implying $\alpha(U) = 0,\ \forall U \subseteq V$, including $\alpha(V_{d,e}) = \alpha^{d,e} = 0$ for $0 \leq d \leq s-r$, which will in turn imply that the functions $g^{x,y}$ above are linearly independent.
   
  Let $H = \{ x : 0 \leq x \leq n, x \equiv k_t \pmod b, t \in [r]\}$. We claim that $H$ has a gap of size at least $s-r+2$. Suppose $n \geq  b+k_1$. Then, since $r(s-r+1) \leq  b-1$, the gap is between $k_1$ and $b+k_1$.  Suppose $(s+k_r \leq )~ n < b+k_1$. Then, the gap is right above $k_r$. This proves the claim. We now need to show that for $T \subseteq V$, $\beta(T)=0$ whenever $dim(T) \notin H$, or whenever $dim(T) \not\equiv k_t \pmod b$,  for any $t \in [r]$. Suppose $v_T$ is the $S$-sized containment vector for $T$. When $dim(T) \not\equiv k_t \pmod b$ for any $t \in [r]$, it follows from the property of the prime $p$ given by Theorem \ref{thm:Zsigmondy} that $\sum\limits_{1 \leq j \leq \scalebox{0.5}{$\qbin{n}{1}{q}$}} v^{1,j}_T - \qbin{k_t}{1}{q} \neq0$ in $\mathbb{F}_p$, for every $t \in [r]$.
   \begin{gather*}
   \beta(T) = \sum\limits_{U\subseteq T} \alpha(U) = \sum\limits_{\substack{dim(U) \leq s-r \\ U \subseteq T}} \alpha(U) = \sum\limits_{\substack{ 0\leq d \leq s-r \\ 1\leq e \leq \scalebox{0.5}{$\qbin{n}{d}{q}$}}} \alpha(V^{d,e})f^{d,e}(v_T)
   \end{gather*}
   Since $\sum\limits_{1 \leq j \leq \scalebox{0.5}{$\qbin{n}{1}{q}$}} v^{1,j}_T - \qbin{k_t}{1}{q} \neq 0$ in $\mathbb{F}_p$ for every $t \in [r]$, $f^{d,e}(v_T) = c(T)g^{d,e}(v_T)$ where $c(T) \neq 0$. Then, 
   \begin{gather*}
    \beta(T) = c(T)\sum\limits_{\substack{ 0\leq d \leq s-r \\ 1\leq e \leq \scalebox{0.5}{$\qbin{n}{d}{q}$}}} \alpha(V^{d,e})g^{d,e}(v_T) = c(T)\sum\limits_{\substack{ 0\leq d \leq s-r \\ 1\leq e \leq \scalebox{0.5}{$\qbin{n}{d}{q}$}}} \alpha^{d,e}g^{d,e}(v_T) = c(T) \cdot 0 = 0. 
   \end{gather*}
   Since the set $H$ and the functions $\alpha$ and $\beta$ satsify the conditions of Lemma \ref{lem:3}, we have $\alpha = 0$. This proves the lemma. 
  \end{proof}

  Recall that we are given a family $\mathcal{F}= \{V_1,V_2,\ldots,V_m\}$ of subspaces of $V$ such that $\forall i \in [m]$, dim($V_i$) $ \equiv k_t \pmod b$, for some $k_t \in K$.  Further, dim($V_i \cap V_j$) $ \equiv \mu_t \pmod b$, for some $\mu_t \in L$ and $K$ and $L$ are disjoint subsets of $\{0,1, \ldots , b-1\}$.
    Let $v_i$ be the containment vector of size $S$ corresponding to subspace $V_i \in \mathcal{F}$. 
    We define the following functions from $\mathbb{F}_2^{S}\rightarrow \mathbb{F}_p$.
    \begin{gather*}
      g^i(v)=g^i(v^{0,1},v^{1,1}, \ldots , v^{1,\scalebox{0.5}{ $\qbin{n}{1}{q}$ }}, \ldots , v^{s,1}, \ldots, v^{s,\scalebox{0.5}{ $\qbin{n}{s}{q}$ }})
      \\ := \prod\limits_{j=1}^{s} \left( \sum_{1\leq y \leq \scalebox{0.6}{$\qbin{n}{1}{q}$}} \left(v_i^{1,y}v^{1,y}\right)\ - \qbin{\mu_j}{1}{q}\right)
    \end{gather*}
Let $v=v_j$. Then, $\sum_{1\leq y \leq \scalebox{0.6}{$\qbin{n}{1}{q}$}} (v_i^{1,y}v^{1,y})$ counts the number of $1$-dimensional subspaces common to $V_i$ and $V_j$. That is, $\sum\limits_{1\leq y \leq \scalebox{0.6}{$\qbin{n}{1}{q}$}} v_i^{1,y}v^{1,y} = \scalebox{0.8}{$\qbin{dim(V_i \cap V_j)}{1}{q}$}$. In $\mathbb{F}_p$, $\scalebox{0.8}{$\qbin{dim(V_i \cap V_j)}{1}{q}$} \neq \scalebox{0.8}{$\qbin{\mu_t}{1}{q}$}$ for any $1 \leq t \leq s$, if $i=j$, \text{ }\text{ } and $\scalebox{0.8}{$\qbin{dim(V_i \cap V_j)}{1}{q}$} = \scalebox{0.8}{$\qbin{\mu_t}{1}{q}$}$ for some $1 \leq t \leq s$ if $i\neq j$. Accordingly, $g^i(v_j) = 
  \begin{cases}
   0,\quad\quad i \neq j \\
   \neq 0 , \quad i=j
  \end{cases}.
$
 \begin{lemma}[Swallowing trick 1]\thlabel{lem:5}
   Let $s+k_r \leq n$ and $r(s-r+1) \leq b-1$. The collection of functions $g^i$, $1 \leq i \leq m$ together with the functions $g^{x,y}$, $0\leq x\leq s-r$, $1 \leq y \leq \qbin{n}{x}{q}$ are linearly independent in $\mathbb{F}_p^{\Omega}$ over $\mathbb{F}_p$.
  \end{lemma}
 \begin{proof}
  Let
  \begin{equation}
    \sum\limits_{1 \leq i \leq m}\alpha^ig^i + \sum\limits_{\substack{0 \leq x \leq s-r \\ 1\leq y \leq \scalebox{0.7}{$\qbin{n}{x}{q}$}}} \alpha^{x,y}g^{x,y} = 0\label{eq:1}
   \end{equation}
 We know that $g^i(v_j)=0$ whenever $i \neq j$, and $g^{x,y}(v_i)=0, 1 \leq i \leq m$. The latter holds because $dim(V_i) \equiv k_t \pmod b$, say equal  to $bl+k_t$, for some $t \in [r]$. Number of $1$-dimensional subspaces in $V_i$ will then be $\qbin{bl+k_t}{1}{q}$ which is equal to $\qbin{k_t}{1}{q}$ in $\mathbb{F}_p$. Suppose we 
   evaluate L.H.S. of Equation \eqref{eq:1} on $v_1$, then
   all terms except the first one vanish. This gives us $\alpha^1 = 0$, and reduces the relation by one term from left. Next, we put $v=v_2$ to get $\alpha^2=0$, and so on. Finally, all $\alpha^i$ terms are zero, and we are left only with functions $g^{x,y}$. These $\alpha^{x,y}$ values are zero from \thref{lem:4}. Therefore, we have shown that \eqref{eq:1} implies that $\alpha^i=0, 1\leq i\leq m$ and $\alpha^{x,y}=0, 0\leq x \leq s-r, 1 \leq y \leq \qbin{n}{x}{q}$, and hence the given functions are linearly independent.
  \end{proof}
\subsection{Proof of Theorem \ref{th:1}: $\left(s+k_r \leq n \mbox { and } r(s-r+1) \leq b-1\right)$}
  \begin{lemma}
  \label{lem_span}
   The collection of functions $f^{x,y}$, $0 \leq x \leq s,\ 1 \leq y \leq \scalebox{0.6}{$\qbin{n}{x}{q}$}$, spans all the functions $g^{x,y}$, $0 \leq x \leq s-r,\ 1 \leq y \leq \scalebox{0.6}{$\qbin{n}{x}{q}$}$ as well as the functions $g^i$, $1 \leq i \leq m$.
  \end{lemma}
  
  \begin{proof}
  Proof in Appendix  \ref{append_proof_of_lem_span}
  \end{proof}
  
  This means that the above functions $g^{x,y}$ and $g^{i}$ belong to the span of functions $f^{x,y}$ which is a function space of dimension at most $S$. From \thref{lem:5}, we know that $g^{x,y}$ and $g^i$ are together linearly independent. Thus, 
  
  \begin{gather*}
   \sum\limits_{j=0}^{s-r} \qbin{n}{j}{q} + m \leq S = \sum\limits_{j=0}^{s} \qbin{n}{j}{q}. \\
   \Rightarrow |\mathcal{F}| = m \leq \qbin{n}{s}{q} + \qbin{n}{s-1}{q} + \cdots + \qbin{n}{s-r+1}{q}.
  \end{gather*}
We thus have the following theorem. 
\begin{theorem}\thlabel{thm:modular_thm_1}
   Let $V$ be a vector space of dimension $n$ over a finite field of size $q$. Let $K = \{k_1, \ldots , k_r\},L = \{\mu_1, \ldots , \mu_s\}$ be two disjoint subsets of $\{0,1, \ldots , b-1\}$ with $k_1 < \cdots < k_r$.  Assume $s+k_r \leq n$ and $r(s-r+1) \leq  b-1$. 
   Let 
   $\mathcal{F} = \{V_1,V_2,\ldots,V_m\}$ be a family of subspaces of $V$ such that $\forall i \in [m]$, dim($V_i$) $\equiv k_t \pmod b$, for some $k_t \in K$; for every distinct $i,j \in [m]$, dim($V_i \cap V_j$) $ \equiv \mu_t \pmod b$, for some $\mu_t \in L$. Moreover, it is given that neither of the following two conditions hold:
   \begin{enumerate}[(i)]
   \item $q+1$ is a power of 2, and $b=2$
   \item $q=2, b=6$
   \end{enumerate}
 Then, $|\mathcal{F}| \leq N(n,s,r,q)$.    		

  \end{theorem}
\subsection{Proof of Theorem \ref{th:1}: unrestricted}
We start with the following lemma  which claims that the functions $g^{x,y}, 0\leq x \leq s-r$, $1\leq y \leq \qbin{n}{x}{q}$, and $\forall t \in [r], x \not\equiv k_t \pmod b$ are linearly independent over $\mathbb{F}_p$. 
\begin{lemma}\thlabel{lem:53}
The collection of functions 
				$\{g^{x,y}\ |\ 0\leq x \leq s-r, 1\leq y \leq \qbin{n}{x}{q},~\mbox{and }\forall t \in [r], x \not\equiv k_t \pmod b \} $
			are linearly independent in the function space $\mathbb{F}_p^{\Omega}$ over $\mathbb{F}_p$.
	\end{lemma}
\begin{proof}
	Recall that $g^{x,y}(v) = f^{x,y}(v) \prod_{t \in [r]}(\sum\limits_{j=1}^{ \scalebox{0.7}{$ \qbin{n}{1}{q}\text{ }$} } v^{1,j} - \qbin{k_t}{1}{q})$. 
Assume, for the sake of contradiction, $\sum\limits_{\substack{0\leq x\leq s-r \\ x \not\equiv k_t \pmod p, \forall t \in [r]}} \alpha^{x,y} g^{x,y} = 0 $ with at least one $\alpha^{x,y}$ as non-zero. Let $\langle x_0,y_0\rangle$ be the first subspace, based on the ordering of subspaces defined in Section \ref{sec:3}, such that  $\alpha^{x_0,y_0}$ is  non-zero. Evaluating both sides on $v_{x_0,y_0}$, we see that all $f^{x,y}$ (and therefore $g^{x,y}$) with $\langle x,y\rangle$ higher in the ordering than $\langle x_0,y_0\rangle$ will vanish (due to the virtue of our ordering), and so we get $\alpha^{x_0,y_0}=0$ which is a contradiction. Here we have crucially used the fact that by ignoring $x \equiv k_t \pmod p$ cases, for any $t \in [r]$, we make sure that $v_{x_0,y_0}$ used above always has $x_0 \not\equiv k_t \pmod b$ and therefore $(\sum\limits_{j=1}^{ \scalebox{0.7}{$ \qbin{n}{1}{q}\text{ }$} } v_{x_0,y_0}^{1,j} - \qbin{k_t}{1}{q}) \not\equiv 0 \pmod p,~\forall t \in [r]$.
	\end{proof}
\begin{lemma}[Swallowing trick 2]\thlabel{lem:swallow_2}
   The collection of functions $g^i$, $1 \leq i \leq m$ together with the functions $g^{x,y}$, $0\leq x \leq s-r, x \not\equiv k_t \pmod b, \forall t \in [r], 1\leq y \leq \qbin{n}{x}{q}$ are linearly independent in $\mathbb{F}_p^{\Omega}$ over $\mathbb{F}_p$.
  \end{lemma}
  \begin{proof}
  Proof is similar to the proof of Lemma \ref{lem:5}. 
  \end{proof}
  Since $s <  b$, for any $0 \leq x \leq s-r$ and for any $t \in [r]$, $x \not \equiv k_t \pmod b$ is equivalent to $x \neq k_t$. Combining Lemmas \ref{lem:53}, \ref{lem:swallow_2} and \ref{lem_span}, we have 
  \begin{equation*}
   \sum\limits_{\substack{0 \leq j \leq s-r, \\ j \neq k_t, t \in [r]}} \qbin{n}{j}{q} + m \leq \sum\limits_{j=0}^{s} \qbin{n}{j}{q}.
  \end{equation*}
  This implies, 
  \begin{equation*}
    |\mathcal{F}| = m \leq = \begin{cases}
                       N(n,s,r,q),\text{ if }s < k_1 + r \\
                       N(n,s,r,q) + \sum\limits_{t \in [r]}\qbin{n}{k_t}{q},\text{ otherwise}
                      \end{cases}
   \end{equation*}
We thus have the following theorem which combined with Theorem \ref{thm:modular_thm_1} yields Theorem \ref{th:1}. 
\begin{theorem}
 Let $V$ be a vector space of dimension $n$ over a finite field of size $q$. Let $K = \{k_1, \ldots , k_r\},L = \{\mu_1, \ldots , \mu_s\}$ be two disjoint subsets of $\{0,1, \ldots , b-1\}$ with $k_1 < \cdots < k_r$.  
  Let 
   $\mathcal{F} = \{V_1,V_2,\ldots,V_m\}$ be a family of subspaces of $V$ such that $\forall i \in [m]$, dim($V_i$) $\equiv k_t \pmod b$, for some $k_t \in K$; for every distinct $i,j \in [m]$, dim($V_i \cap V_j$) $ \equiv \mu_t \pmod b$, for some $\mu_t \in L$. Moreover, it is given that neither of the following two conditions hold:
   \begin{enumerate}[(i)]
   \item $q+1$ is a power of 2, and $b=2$
   \item $q=2, b=6$
   \end{enumerate}
 Then, 
 \begin{equation*}
    		|\mathcal{F}| \leq \begin{cases}
    											\scalebox{0.7}{ $N(n,s,r,q)$}, \hskip 1cm \mbox{ if } (s < k_1 + r)  \\
    											\scalebox{0.7}{ $N(n,s,r,q) + \sum\limits_{t \in [r]}\qbin{n}{k_t}{q}$}, \hskip 1cm \text{otherwise.}
											\end{cases}    		
    		\end{equation*}
 \end{theorem}

\begin{center}\section*{PART B: Fractional L-intersecting families of subspaces}\end{center}
 Let $L=\{\frac{a_1}{b_1}, \ldots , \frac{a_s}{b_s}\}$ be a collection of positive irreducible fractions, each strictly less than $1$. Let $V$ be a vector space of dimension $n$ over a finite field of size $q$. Let $\mathcal{F}$ be a family of subspaces of $V$. Recall that, we call $\mathcal{F}$ a  \emph{fractional $L$-intersecting family of subspaces} if $\forall A, B \in \mathcal{F}$, $dim(A \cap B) \in \{\frac{a_i}{b_i}dim(A), \frac{a_i}{b_i}dim(B)\}$, for some $\frac{a_i}{b_i} \in L$. In Section \ref{sec:frac_gen_bound}, we prove a general upper bound for the size of a fractional $L$-intersecting family using Theorem \ref{th:1} proved in Part A. In Section \ref{sec:frac_singleton_bound}, we improve this upper bound for the special case when $L = \{\frac{a}{b}\}$ is a singleton set with $b$ being a prime number. 
 \section{A general upper bound} \label{sec:frac_gen_bound}
  The key idea we use here is to split the fractional $L$ intersecting family $\mathcal{F}$ into subfamilies and then use \thref{th:1} to bound each of them. 
\begin{lemma}\thlabel{lem:51}
    Let $L = \{\frac{a_1}{b_1}, \frac{a_2}{b_2}, \ldots , \frac{a_s}{b_s}\}$, where every $\frac{a_i}{b_i}$ is a irreducible fraction in the open interval $(0,1)$. Let $\mathcal{F} = \{V_1, \ldots , V_m\}$ be a fractional $L$-intersecting family of subspaces of a vector space $V$ of dimension $n$ over a finite field of size $q$. Let $\mathcal{F}^p_k$ denote subspaces in $\mathcal{F}$ whose dimensions leave a remainder $k \pmod p$, where $p$ is a prime number. That is, $\mathcal{F}^p_k := \{W \in \mathcal{F}\ |\ dim(W) \equiv k \pmod p\}$. 
    Further, let $k>0$, and $p > \max(b_1,b_2,\ldots,b_s)$. 
    Then, 
    \begin{equation*}
    |\mathcal{F}_k^p| \leq \begin{cases}
                       \qbin{n}{s}{q},\text{ if }(2p \leq n+2) \text{ or }(s < k+1)  \\
                       \qbin{n}{s}{q} + \qbin{n}{k}{q},\text{ otherwise.}
                      \end{cases}
   \end{equation*}
    
   
    \end{lemma}
    
    \begin{proof}
		Apply \thref{th:1} with family $\mathcal{F}$ replaced by $\mathcal{F}_k^p$, $K = \{k\}$, $r=1$, $b$ replaced by $p$, and each $\mu_i$ replaced by $\frac{a_i}{b_i}k$. Let $s'$ ($\leq s$) be the number of distinct $\mu_i$'s.  Notice that $k > 0$, and $p > b_i > a_i$ ensure that $k \not\equiv \frac{a_i}{b_i}k \pmod p$. Moreover, since $s'\leq p-1$ and $k \leq p-1$, we have $s'+k \leq n$ if $2p \leq n+2$. Since $p>b_i$ and every $b_i \geq 2$, we have $p>2$. This avoids  the bad cases (i) and (ii) of Theorem \ref{th:1}. Thus, we satisfy the premise of \thref{th:1} and the conclusion follows.
    \end{proof}
 
 Suppose $2p \leq n+2$. The above lemma immediately gives us a bound of $| \mathcal{F} | \leq |\mathcal{F}_0^p| + (p-1) \qbin{n}{s}{q}$. But it could be that most subspaces belong to $\mathcal{F}_0^p$. To overcome this problem, we instead choose a set of primes $P$ such that no subspace can belong to $\mathcal{F}_0^p$ for every $p \in P$. A natural choice is to take just enough primes in increasing order so that the product of these primes exceeds $n$, because then any subspace with dimension divisible by all primes in $P$ will have a dimension greater than $n$, which is not possible.
 
 \begin{lemma}\thlabel{th:52}
  Let $L = \{\frac{a_1}{b_1}, \frac{a_2}{b_2}, \ldots , \frac{a_s}{b_s}\}$, where every $\frac{a_i}{b_i}$ is an irreducible fraction in the open interval $(0,1)$. Let $\mathcal{F} = \{V_1, \ldots , V_m\}$ be a fractional $L$-intersecting family of subspaces of a vector space $V$ of dimension $n$ over a finite field of size $q$.
  Let $t := \max(b_1,b_2,\ldots,b_s)$ and $g(t,n) := \frac{2(2t+\ln n)}{\ln(2t+\ln n)}$. Suppose $2g(t,n)\ln(g(t,n)) \leq n+2$. Then,
  \begin{align*}
  |\mathcal{F}| \leq 2g^2(t,n) \ln(g(t,n)) \qbin{n}{s}{q}
  \end{align*}
  \end{lemma}
  \begin{proof}
  For some $\beta$ to be chosen later, choose $P$ to be the set $\{p_{\alpha+1},p_{\alpha+2},\ldots,p_{\beta}\}$ where $p_l$ denotes the $l^{th}$ prime number and $p_\alpha \leq t <p_{\alpha+1}<p_{\alpha+2}<\cdots<p_\beta$. Let $l\#$ denote the product of all primes less than or equal to $l$. Thus, $p_l\#$ which is known as the \emph{primorial function}, is the product of the first $l$ primes. It is known that $p_l\# = e^{(1+o(1))l\ln l}$ and $l\# = e^{(1+o(1))l}$. We require the following condition for the set $P$:
    \begin{equation*}
	\frac{p_\beta\#}{t\#} > n 
  \end{equation*}
Using the bounds  for $p_l\#$ and $l\#$ discussed above, we find that it is sufficient to choose $\beta \geq \frac{2(2t+\ln n)}{\ln(2t+\ln n)} := g(t,n)$. Let $ \beta = g(t,n)$. From the Prime Number Theorem, it follows that $p_{\beta}$ (and so $p_{\alpha+1}, p_{\alpha+2},\ldots, p_{\beta-1}$ as well) is at most $2g(t,n)\ln(g(t,n))$. 
We are given that $2p \leq 2p_{\beta} \leq n+2$, for every $p \in P$. We apply \thref{lem:51} with $p=p_{\alpha+1}$ to get 
 \begin{equation*}
  |\mathcal{F}| \leq |\mathcal{F}_0^{p_{\alpha+1}}| + (p_{\alpha+1} -1)\qbin{n}{s}{q}
	\end{equation*}
	
	Next, apply \thref{lem:51} on $\mathcal{F}_0^{p_{\alpha+1}}$ with $p=p_{\alpha + 2}$ and so on. As argued above, no subspace is left uncovered after we reach $p_\beta$. This means,
	
	\begin{align*}
		|\mathcal{F}| & \leq (p_{\alpha+1} + p_{\alpha+2} + \cdots + p_\beta - (\beta - \alpha) ) 	\qbin{n}{s}{q} \\
		& <(\beta-\alpha) p_\beta \qbin{n}{s}{q} \\
		& < \beta p_\beta \qbin{n}{s}{q} \\
		& \leq 2 g^2(t,n) \ln(g(t,n)) \qbin{n}{s}{q}
		\end{align*}
	\end{proof}

	\begin{lemma}\thlabel{th:55}
		Let $L = \{\frac{a_1}{b_1}, \frac{a_2}{b_2}, \ldots , \frac{a_s}{b_s}\}$, where every $\frac{a_i}{b_i}$ is an irreducible fraction in the open interval $(0,1)$. Let $\mathcal{F} = \{V_1, \ldots , V_m\}$ be a fractional $L$-intersecting family of subspaces of a vector space $V$ of dimension $n$ over a finite field of size $q$.
  Let $t := \max(b_1,b_2,\ldots,b_s)$ and $g(t,n) := \frac{2(2t+\ln n)}{\ln(2t+\ln n)}$. Then,
  
  \begin{align*}
  |\mathcal{F}| \leq 2g^2(t,n) \ln(g(t,n)) \qbin{n}{s}{q} + g(t,n) \sum\limits_{i=1}^{s-1} \qbin{n}{i}{q}
  \end{align*}
	\end{lemma}
 
 	\begin{proof}
 	Let $P = \{p_{\alpha+1},p_{\alpha+2},\ldots,p_{\beta}\}$, where $\beta = g(t,n)$ and $p_\beta \leq 2g(t,n)\ln(g(t,n))$.  
 		The proof is similar to the proof of Lemma \ref{th:52}. We apply Lemma \ref{lem:51} with $p = p_{\alpha + 1}$ to show that 
 		
 		\begin{align*}
 		|\mathcal{F}| \leq |\mathcal{F}_0^{p_{\alpha + 1}}| +(p_{\alpha + 1}-1)\qbin{n}{s}{q} + \sum\limits_{i=1}^{s-1} \qbin{n}{i}{q}
 		\end{align*}
 Next, we apply Lemma  \ref{lem:51} on $\mathcal{F}_0^{p_{\alpha + 1}}$ with $p = p_{\alpha + 2}$ and so on as shown in the proof of Lemma \ref{th:52} to get the desired bound. 		
 		\begin{align*}
 		|\mathcal{F}| & \leq (p_{\alpha+1}+p_{\alpha+2}+\cdots+p_{\beta} - (\beta-\alpha))\qbin{n}{s}{q} + (\beta-\alpha)\sum\limits_{i=1}^{s-1} \qbin{n}{i}{q}\\
 		& <(\beta-\alpha) \Big( p_\beta \qbin{n}{s}{q} + \sum\limits_{i=1}^{s-1} \qbin{n}{i}{q} \Big) \\
 		& < \beta \Big( p_\beta \qbin{n}{s}{q} + \sum\limits_{i=1}^{s-1} \qbin{n}{i}{q} \Big) \\
 		& \leq 2 g^2(t,n) \ln(g(t,n)) \qbin{n}{s}{q} + g(t,n) \sum\limits_{i=1}^{s-1} \qbin{n}{s}{q}
 		\end{align*}
 	\end{proof}
 
Since $p_{\alpha + 1} > t$, we have  $p_{\alpha+1}p_{\alpha+2}\cdots p_\beta > t^{\beta-\alpha}$. This implies that, if $t^{\beta-\alpha} \geq n$, then the product of the primes in $P$ will be greater than $n$ as desired. Substituting $\beta - \alpha$ with $\frac{\ln n}{\ln t}$ (and $p_\beta$ with $2g(t,n)\ln(g(t,n))$) in the second inequality above, we get another upper bound of $|\mathcal{F}| \leq 2 g(t,n) \frac{\ln(n) \ln(g(t,n))}{\ln t} \qbin{n}{s}{q} + \frac{\ln n}{\ln t} \sum\limits_{i=1}^{s-1} \qbin{n}{i}{q}$. We can do a similar substitution for $\beta - \alpha$ in the calculations done at the end of the proof of \thref{th:52} to get a similar bound. 
 	
Combining all the results in this section, we get Theorem \ref{thm:frac_general_thm}
 	
 	
 	
 	
 	
 	
 	
 	


 \begin{section}{An improved bound for singleton L}
\label{sec:frac_singleton_bound}  
  In this section, we improve the upper bound for the size of a fractional $L$-intersecting family obtained in \thref{thm:frac_general_thm} for the special case $L = \{ \frac{a}{b}\}$, where $b$ is  a constant prime. 
  Before we give the proof, below  we restate the the statement of Theorem \ref{th:61}. \vspace{0.1in} \\ 
     \textbf{Statement of Theorem \ref{th:61}}:  
     Let $L=\{\frac{a}{b}\}$, where $\frac{a}{b}$ is a positive irreducible fraction less than $1$ and $b$ is a prime. Let  $\mathcal{F}$ be a fractional $L$-intersecting family of subspaces of a vector space $V$ of dimension $n$ over a finite field of size $q$. Then, we have $|\mathcal{F}| \leq (b-1)( \qbin{n}{1}{q} +1) \lceil \frac{\ln n}{\ln b} \rceil + 2$.
  
   

    

  \begin{proof}
    
    We assume that all the subspaces in the family except possibly one subspace, say $W$, have a dimension divisible by $b$. Otherwise, $\mathcal{F}$ cannot satify the property of a fractional $\frac{a}{b}$-intersecting family. Let us ignore $W$ in the discussion to follow. For any subspace $V_i$ that is not the zero subspace, let $k$ be the largest power of $b$ that divides $dim(V_i)$. Then, $dim(V_i) = rb^{k+1} + jb^k$, for some $1 \leq j < b, r\geq 0$. Consider the subfamily,
    \begin{equation*}
	\mathcal{F}^{j,k} = \{ V_i ~:~ b^k | dim(V_i),~ b^{k+1}\not| dim(V_i),~ dim(V_i)=rb^{k+1} + j b^k\text{ for some } r\geq 0, j \in [b]\}
    \end{equation*}
    The subfamily $\mathcal{F}^{j,k}, 1\leq k \leq \lceil \frac{\ln n}{\ln b} \rceil, 1 \leq j < b$, cover each and every subspace (except the zero subspace and the subspace $W$) of $\mathcal{F}$ exactly once. We will show that $|\mathcal{F}^{j,k}| \leq \qbin{n}{1}{q}+1$, which when multiplied with the number of values $j$ and $k$ can take will immediately imply the theorem. 
  
   Let $m^{j,k} = |\mathcal{F}^{j,k}|$. Let $M^{j,k}$ be an $m^{j,k} \times \qbin{n}{1}{q}$ $0$-$1$ matrix whose rows correspond to the subspaces of $\mathcal{F}^{j,k}$ in any given order, whose columns correspond to the $1$-dimensional subspaces of $V$ in any given order, and the $(i\text{-}l)^{th}$ entry is $1$ if and only if the $i^{th}$ subspace of $\mathcal{F}^{j,k}$ contains the $l^{th}$ $1$-dimensional subspace. Let $N^{j,k} = M^{j,k}\cdot(M^{j,k})^T$. Any diagonal entry $N^{j,k}_{i,i}$ is the number of $1$-dimensional subspaces in the $i^{th}$ subspace in $\mathcal{F}^{j,k}$, and an off-diagonal entry $N^{j,k}_{i,l}$ is number of $1$-dimensional subspaces common to the $i^{th}$ and $l^{th}$ subspaces of $\mathcal{F}^{j,k}$.
   \begin{align*}
    N^{j,k}_{i,i} = \qbin{r_1b^{k+1} + jb^k}{1}{q} = \qbin{b^{k-1}}{1}{q}\qbin{r_1b^2 + jb}{1}{q^{b^{k-1}}},\\
    N^{j,k}_{i,l} = \qbin{r_2ab^{k} + jab^{k-1}}{1}{q} = \qbin{b^{k-1}}{1}{q}\qbin{r_2ab + ja}{1}{q^{b^{k-1}}},
   \end{align*}
   for some $r_1,r_2$ (may be different for different values of $i,l$). Let $P^{j,k}$ be the matrix obtained by dividing each entry of $N^{j,k}$ by $\qbin{b^{k-1}}{1}{q}$.
   \begin{equation*}
    \det(N^{j,k}) = \qbin{b^{k-1}}{1}{q}^{m^{j,k}} \det(P^{j,k})
   \end{equation*}
We will show that $\det(P^{j,k})$ is non-zero, thereby implying $\det(N^{j,k})$ is also non-zero. Consider $\det(P^{j,k}) \pmod {\qbin{b}{1}{q^{b^{k-1}}}}$.
   
   \begin{align*}
    P^{j,k}_{i,i} \equiv \qbin{r_1b^2 + jb}{1}{q^{b^{k-1}}} \pmod {\qbin{b}{1}{q^{b^{k-1}}}} \equiv 0 \pmod {\qbin{b}{1}{q^{b^{k-1}}}}, \\
    P^{j,k}_{i,l} \equiv \qbin{r_2ab + ja}{1}{q^{b^{k-1}}} \pmod {\qbin{b}{1}{q^{b^{k-1}}}} \equiv \qbin{r_3}{1}{q^{b^{k-1}}} \pmod {\qbin{b}{1}{q^{b^{k-1}}}},
   \end{align*}
  where $r_3 = ja  \mod b$ and $1 \leq r_3 \leq b-1$ (since $j<b,a<b$, and $b$ is a prime, we have $1 \leq r_3 \leq b-1$).
  We know that the determinant of an $r \times r$ matrix where diagonal entries are $0$ and off-diagonal entries are all $1$ is $(-1)^{r-1}(r-1)$.
  \begin{equation*}
   \det(P^{j,k}) \equiv (\qbin{r_3}{1}{q^{b^{k-1}}})^{m^{j,k}} (-1)^{m^{j,k}-1} (m^{j,k}-1) \pmod {\qbin{b}{1}{q^{b^{k-1}}}}
  \end{equation*}
 Let $Q^{j,k}$ be the matrix formed by taking all but the last row and the last column  of $P^{j,k}$.
 \begin{equation*}
   \det(Q^{j,k}) \equiv (\qbin{r_3}{1}{q^{b^{k-1}}})^{m^{j,k}-1} (-1)^{m^{j,k}-2} (m^{j,k}-2) \pmod {\qbin{b}{1}{q^{b^{k-1}}}}
  \end{equation*}
 We will now  show that one of $\det(P^{j,k})$ or $\det(Q^{j,k})$ is non-zero $\pmod {\qbin{b}{1}{q^{b^{k-1}}}}$ and therefore non-zero in $\mathbb{R}$. First, we show that $\qbin{r_3}{1}{q^{b^{k-1}}}^{m^{j,k}}$ is not divisible by $\qbin{b}{1}{q^{b^{k-1}}}$. 
  Suppose $s_3 \equiv r_3^{-1} \pmod b$.
   \begin{gather*}
    \qbin{r_3}{1}{q^{b^{k-1}}}^{m^{j,k}} \qbin{s_3}{1}{q^{r_3b^{k-1}}}^{m^{j,k}} =\qbin{r_3s_3}{1}{q^{b^{k-1}}}^{m^{j,k}} \\
    \qbin{r_3s_3}{1}{q^{b^{k-1}}}^{m^{j,k}} \equiv \qbin{1}{1}{q^{b^{k-1}}}^{m^{j,k}} \pmod {\qbin{b}{1}{q^{b^{k-1}}}} \equiv 1 \pmod {\qbin{b}{1}{q^{b^{k-1}}}}
  \end{gather*}
  Therefore, $\qbin{r_3}{1}{q^{b^{k-1}}}^{m^{j,k}}$ is invertible modulo $\qbin{b}{1}{q^{b^{k-1}}}$, and hence the former is not divisible by the latter. Suppose $\qbin{r_3}{1}{q^{b^{k-1}}}^{m^{j,k}} (-1)^{m^{j,k}-1} (m^{j,k}-1)$ is divible by $\qbin{b}{1}{q^{b^{k-1}}}$. 
  We may ignore $(-1)^{m^{j,k}-1}$ for divisibility purpose. Then,  there must be a product of prime powers that is equal to $(m^{j,k}-1)$ multiplied by $\qbin{r_3}{1}{q^{b^{k-1}}}^{m^{j,k}}$ such that this product is divisible by $\qbin{b}{1}{q^{b^{k-1}}}$.
  Observe that, $\qbin{r_3}{1}{q^{b^{k-1}}}^{m^{j,k}-1}$ has only lesser powers of the same primes, and $m^{j,k}-1$ and $m^{j,k}-2$ cannot have any prime in common. So, the product $\qbin{r_3}{1}{q^{b^{k-1}}}^{m^{j,k}-1} (m^{j,k}-2)$ cannot be divisible by $\qbin{b}{1}{q^{b^{k-1}}}$, which is what we wanted to prove.
  
  Therefore, either $P^{j,k}$ or $Q^{j,k}$ is a full rank matrix, or $rank(P^{j,k}) \geq m^{j,k}-1$. Being a non-zero multiple of $P^{j,k}$, $rank(N^{j,k}) \geq m^{j,k}-1$. But we know that $rank(AB) \leq \min(rank(A),rank(B))$, for any two matrices $A,B$.
  
  \begin{align*}
   m^{j,k}-1 \leq rank(N^{j,k}) & \leq \min(rank(M^{j,k}),rank((M^{j,k})^T)) \\
   & = rank(M^{j,k}) \\
   & \leq \qbin{n}{1}{q}
  \end{align*}
Or, $m^{j,k} \leq \qbin{n}{1}{q} + 1$, as required. It follows that, \begin{equation*}
   |\mathcal{F}| = m \leq 2 + \sum\limits_{\substack{ 1 \leq k \leq \lceil \frac{\ln n}{\ln b} \rceil \\ 1 \leq j < b}} m^{j,k} \leq (b-1)(\qbin{n}{1}{q}+1) \left\lceil \frac{\ln n}{\ln b} \right\rceil + 2.
  \end{equation*}

  \end{proof}
\end{section}
\section{Concluding remarks}
In \thref{th:1}, for $|\mathcal{F}|$ to be at most $N(n,s,r,q)$, one of the necessary conditions is $r(s-r+1) \leq b-1$. When $r=1$, this  condition is always true as $L \subseteq \{0, 1, \ldots , b-1\}$. However, when $r \geq 2$, it is not the case. Would it be possible to get the same upper bound for $|\mathcal{F}|$ without having to satisfy such a strong necessary condition? Another interesting question concerning Theorem \ref{th:1} is regarding its tightness. From Example \ref{exam:frankl_graham}, we know that \thref{th:1} is tight when $r=1$. However, since Theorem \ref{th:1} requires the sets $K$ and $L$ to be disjoint it is not possible to extend the  construction in Example \ref{exam:frankl_graham} to obtain a tight example for the case $r \geq 2$. Further, we know of no other tight example for this case. Therefore, we are not clear whether Theorem \ref{th:1} is tight when $r \geq 2$.      

We believe that the upper bounds given by Theorems \ref{thm:frac_general_thm} and \ref{th:61} are not tight. Proving tight upper bounds in both the scenarios is a question that is obviously interesting. One possible approach to try would be to answer the following simpler question. 
Consider the case when $L=\{\frac{1}{2}\}$. We call such a family a \emph{bisection-closed family} of subspaces. Let $\mathcal{F}$ be a bisection closed family of subspaces of a vector space $V$ of dimension $n$ over a finite field of size $q$. From Theorem \ref{th:61}, we know that $|\mathcal{F}| \leq (\qbin{n}{1}{q} + 1)\log_2n + 2$.  We believe that $|\mathcal{F}| \leq c\qbin{n}{1}{q}$, where $c$ is a constant. 
Example \ref{exam:frac_spe} gives a `trivial' bisection-closed family of size $\qbin{n-1}{1}{q}$ where every subspace contains the vector $v_1$. It would be interesting to look for non-trivial examples of large bisection-closed families. 
\bibliographystyle{apalike}

\appendix
\section{Proof of Proposition \ref{prop:1}}\label{append_proof_of_prop_inversion_formula}
\begin{proof}
(Forward direction) $\beta(W) = \sum\limits_{U\subseteq W} \alpha(U)$ is given.
    \begin{align*}
     \sum_{ \substack{U\subseteq W \\ d= dim(W)-dim(U)}} (-1)^d q^{\frac{d(d-1)}{2}} \beta(U) = \sum_{ \substack{U\subseteq W \\ d= dim(W)-dim(U)}} \left( (-1)^d q^{\frac{d(d-1)}{2}} \sum\limits_{U'\subseteq U}\alpha(U')\right) \\
      = \sum\limits_{U' \subseteq W} \left( \alpha(U') \Big( \sum_{\substack{ U' \subseteq U \subseteq W \\ d = dim(W) - dim(U)}} (-1)^d q^{\frac{d(d-1)}{2}}\Big) \right) \\
     = \sum\limits_{U' \subseteq W} \left( \alpha(U') \Big( 
     \sum\limits_{d=0}^{dim(W)-dim(U')}\left(\qbin{dim(W)-dim(U')}{d
     }{q} (-1)^d q^{\frac{d(d-1)}{2}}\Big)\right) \right)
    \end{align*}
  
  We know that $\sum\limits_{d=0}^{n}\qbin{n}{d}{q} (-1)^d q^{\frac{d(d-1)}{2}} = 0$ unless $n=0$, and $1$ otherwise. Therefore, only $dim(W)-dim(U')=0$ or $U'=W$ term will survive in the summation above, giving,
  
      \begin{align*}
	\sum_{ \substack{U\subseteq W \\ d= dim(W)-dim(U)}} (-1)^d q^{\frac{d(d-1)}{2}} \beta(U) = \alpha(W)
      \end{align*}
    \\
    (Reverse direction) $\alpha(W) = \sum\limits_{ \substack{U\subseteq W \\ d= dim(W)-dim(U)}} (-1)^d q^{\frac{d(d-1)}{2}} \beta(U)$ is given.
    
    \begin{align*}
     \sum\limits_{U\subseteq W} \alpha(U) = \sum\limits_{U \subseteq W}\left( \sum\limits_{ \substack{U'\subseteq U \\ d= dim(U)-dim(U')}} (-1)^d q^{\frac{d(d-1)}{2}} \beta(U') \right) \\
     = \sum\limits_{U' \subseteq W}\left( \beta(U') \sum\limits_{ \substack{U'\subseteq U \subseteq W\\ d= dim(U)-dim(U')}} (-1)^d q^{\frac{d(d-1)}{2}} \right) \\
     = \sum\limits_{U' \subseteq W}\left( \beta(U') \sum\limits_{ d =0}^{dim(W)-dim(U')} \left( (-1)^d q^{\frac{d(d-1)}{2}} \qbin{dim(W)-dim(U')}{d}{q} \right) \right)
    \end{align*}
    
    The only term that will survive in the above summation is the term when $dim(W)-dim(U')=0$ or $U'=W$.
    This proves $\sum\limits_{U\subseteq W} \alpha(U) = \beta(W)$. 
\end{proof}    
\section{Proof of Proposition \ref{prop:mob_inv_gen}}\label{append_proof_of_prop_mob_inv_gen}
\begin{proof}
We know that $\beta(W) = \sum_{U\subseteq W} \alpha(U)$.
     
     \begin{gather*}
      \sum_{\substack{T :\ W \subseteq T \subseteq Y \\ d=dim(Y)-dim(T)}} (-1)^d q^{\frac{d(d-1)}{2}} \beta(T) = \sum_{\substack{T :\ W \subseteq T \subseteq Y \\ d=dim(Y)-dim(T)}} \left( (-1)^d q^{\frac{d(d-1)}{2}} \sum\limits_{U \subseteq T} \alpha(U) \right) \\
      = \sum\limits_{U\subseteq Y} \left( \alpha(U) \sum\limits_{\substack{T: U \subseteq T, \\ W \subseteq T \subseteq Y, \\ d= dim(Y)-dim(T)}} (-1)^d q^{\frac{d(d-1)}{2}} \right)
     \end{gather*}
  
    Let $W'$ be the span of union (of vectors) of $U$ and $W$. From our notation above, $W' = U \cup W$. As $T$ is a superspace of both $U$ and $W$, $T$ contains all the vectors of $U$ and $W$. Thus, $W'\subseteq T$. 
    
    \begin{gather*}
      \sum_{\substack{T :\ W \subseteq T \subseteq Y \\ d=dim(Y)-dim(T)}} (-1)^d q^{\frac{d(d-1)}{2}} \beta(T) = \sum\limits_{U\subseteq Y} \left( \alpha(U) \sum\limits_{\substack{T:W \cup U \subseteq T \subseteq Y \\ d= dim(Y)-dim(T)}} (-1)^d q^{\frac{d(d-1)}{2}} \right) \\
      = \sum\limits_{U\subseteq Y} \left( \alpha(U) \sum\limits_{d=0}^{dim(Y)-dim(W')} \left( (-1)^d q^{\frac{d(d-1)}{2}} \qbin{dim(Y)-dim(W')}{d}{q} \right) \right)
     \end{gather*}
    Again, the only non-zero term in the second summation above is when $dim(Y)-dim(W')=0$, or $Y=W'$.
    
    \begin{gather*}
     \sum_{\substack{T :\ W \subseteq T \subseteq Y \\ d=dim(Y)-dim(T)}} (-1)^d q^{\frac{d(d-1)}{2}} \beta(T) = \sum\limits_{U: U \cup W = Y} \alpha(U)
    \end{gather*}
\end{proof}
\section{Proof of Lemma  \ref{lem_span}}
\label{append_proof_of_lem_span}
\begin{proof}
Let $v \in \mathbb{F}_2^S$. The key observation here is that the product $f^{x,y}(v)f^{1,z}(v), 0\leq x \leq s-1, 1 \leq y \leq \qbin{n}{x}{q}, 1\leq z \leq \qbin{n}{1}{q}$ may be replaced by the function $f^{x',w}(v)$, where $x\leq x' \leq x+1, 1\leq w\leq \qbin{n}{x'}{q}$.   If $V_{1,z} \subseteq V_{x,y}$, it is trivial that $f^{x,y}(v)f^{1,z}(v)=f^{x,y}(v)$, since $f^{x,y}(v)=1$ only if $f^{1,z}(v)=1$.    If $V_{1,z} \not\subseteq V_{x,y}$, we let $V_{x',w}$ be the span of union of vectors of $V_{1,z}$ and $V_{x,y}$. Suppose, a vector space $U$ contains both $V_{1,z}$ and $V_{x,y}$. Then, it is clear that it must contain the span of their union as well. Similarly, a vector space $U$ that does not contain either $V_{1,z}$ or $V_{x,y}$, cannot contain $V_{x',w}$.
   Thus, $f^{x,y}(v)f^{1,z}(v) = f^{x',w}(v)$. To see why $x' =x+1$ (in case $V_{1,z} \not\subseteq V_{x,y})$, the space $V_{x',w}$ may be obtained by taking any (non-zero) vector of $V_{1,z}$ and introducing it into the basis of $V_{x,y}$. The space spanned by this extended basis is exactly $V_{x',w}$ by definition, and the size of basis has increased by exactly 1.
  
    By induction, it follows that,
    
    \begin{equation*}
     f^{1,y_1}(v)f^{1,y_2}(v) \cdots f^{1,y_l}(v) = f^{x,y}(v)
    \end{equation*}
    for some $x,y$ where, $1\leq x \leq l, 1\leq y\leq \qbin{n}{x}{q}$. That is, a product of $l$ functions of the form $f^{1,y}$ may be replaced by a single function $f^{x,y}$ where $x$ is at most $l$.
    
    Now consider functions
    
    \begin{gather*}
      g^i(v)=g^i(v^{0,1},v^{1,1}, \cdots , v^{1,\scalebox{0.5}{ $\qbin{n}{1}{q}$ }}, \cdots , v^{s,1}, \cdots, v^{s,\scalebox{0.5}{ $\qbin{n}{s}{q}$ }})
      \\ = \prod\limits_{j=1}^{s} \left( \sum_{1\leq y \leq \scalebox{0.6}{$\qbin{n}{1}{q}$}} \left(v_i^{1,y}v^{1,y}\right)\ - \qbin{\mu_j}{1}{q}\right)
      \\ = \prod\limits_{j=1}^{s} \left( \sum_{1\leq y \leq \scalebox{0.6}{$\qbin{n}{1}{q}$}} \left(v_i^{1,y}f^{1,y}(v)\right)\ - \qbin{\mu_j}{1}{q}\right)
    \end{gather*}
    
    Since the functions $f^{x,y}$ only take $0/1$ values, we can reduce any exponent of 2 or more on the function after expanding the product to 1. Moreover, the terms will all be products of the form $f^{1,y_1}f^{1,y_2} \cdots f^{1,y_l}(v), 1\leq l \leq s$. These are replaced according to the observation above by single function of the form $f^{x,y}(v)$, and thus the set of function $f^{x,y}, 0\leq x \leq s, 1 \leq y \leq \qbin{n}{x}{q}$ span all functions $g^i(v)$. Note that $f^{0,1}(v)$ is the constant function $1$.
  
    Similarly, for $0\leq x \leq s-r, 1\leq y \leq \qbin{n}{x}{q}$,
    
    \begin{align*}
      g^{x,y}(v) = f^{x,y}(v)\prod\limits_{t \in [r]}\left(\sum\limits_{j=1}^{ \scalebox{0.7}{$ \qbin{n}{1}{q}\text{ }$} }v^{1,j} - \qbin{k_t}{1}{q}\right)\\
      = f^{x,y}(v)\prod\limits_{t \in [r]}\left(\sum\limits_{j=1}^{ \scalebox{0.7}{$ \qbin{n}{1}{q}\text{ }$} }f^{1,j}(v) - \qbin{k_t}{1}{q}\right)\\
     = f^{x,y}(v)\left(\sum\limits_{x'=0}^r\sum\limits_{y'=1}^{\qbin{n}{x'}{q}} c_{x',y'}f^{x',y'}(v)\right)\tag{$c_{x',y'}$ are constants}\\
     = \sum\limits_{x'=0}^s\sum\limits_{y'=1}^{\qbin{n}{x'}{q}} c_{x',y'}f^{x',y'}(v)\tag{$c_{x',y'}$ are constants}
    \end{align*}
    Thus, the set of function $f^{x,y}, 0\leq x \leq s, 1 \leq y \leq \qbin{n}{x}{q}$ span all functions $g^{x,y}(v), 0\leq x \leq s-r,1\leq y \leq \qbin{n}{x}{q}$.
\end{proof}
\end{document}